%% file: BeautifulDrift.tex
\documentclass[a4paper,USenglish,cleveref,autoref,thm-restate,pdfa]{lipics-v2021}
\hideLIPIcs
\nolinenumbers

\usepackage{tikz}
\usepackage{xspace}
\usepackage{upgreek}
\usepackage{todonotes}
\usepackage{mathtools}

\title{Intuitive Analyses via Drift Theory}
 \titlerunning{}
\author{Andreas G\"obel} % Name
    {Hasso Plattner Institute, University of Potsdam, Potsdam, Germany \and \url{https://hpi.de/friedrich/people/andreas-goebel.html}} % Affiliation and homepage
    {andreas.goebel@hpi.de} % e-mail
    {} % ORC-ID
    {} % Funding
\author{Timo K\"otzing} % Name
    {Hasso Plattner Institute, University of Potsdam, Potsdam, Germany \and \url{https://hpi.de/friedrich/people/timo-koetzing.html}} % Affiliation and homepage
    {timo.koetzing@hpi.de} % e-mail
    {} % ORC-ID
    {} % Funding
\author{Martin S. Krejca} % Name
    {LIX, CNRS, Ecole Polytechnique, IP Paris, Paris, France \and \url{https://www.lix.polytechnique.fr/Labo/Martin.KREJCA/}} % Affiliation and homepage
    {martin.krejca@polytechnique.edu} % e-mail
    {https://orcid.org/0000-0002-1765-1219} % ORC-ID
    {} % Funding

\authorrunning{A. G\"obel and T. K\"otzing and M. S. Krejca} %mandatory. First: Use abbreviated first/middle names.
\Copyright{Andreas G\"obel and Timo K\"otzing and Martin S. Krejca}%mandatory, please use full first names. OASIcs
% mandatory: Please provide 1-5 keywords

\EventEditors{John Q. Open and Joan R. Access}
\EventNoEds{2}
\EventLongTitle{30th Annual European Symposium on Algorithms (ESA 2022)}
\EventShortTitle{ESA 2022}
\EventAcronym{ESA}
\EventYear{2022}
\EventDate{September 5--9, 2022}
\EventLocation{Potsdam, Germany}
\EventLogo{}
\SeriesVolume{42}
\ArticleNo{23}
\input{commands}

\input{additional_commands}

\begin{document}

\maketitle

\begin{abstract}
%Humans are bad with probabilities, and the analysis of randomized algorithms offers many pitfalls for the human mind.
\emph{Drift theory} is an intuitive tool for reasoning about random processes: It allows turning expected stepwise changes into expected first-hitting times. While drift theory is used extensively by the community studying randomized search heuristics, it has seen hardly any applications outside of this field, in spite of many research questions that can be formulated as first-hitting times.

We state the most useful drift theorems and demonstrate their use for various randomized processes, including the coupon collector process, winning streaks, approximating vertex cover, and a random sorting algorithm. We also consider processes without expected stepwise change and give theorems based on drift theory applicable in such scenarios. We use these theorems for the analysis of the gambler's ruin process, for a coloring algorithm, for an algorithm for 2-SAT, and for a version of the Moran process without bias. A final tool we present is a tight theorem for processes on finite state spaces, which we apply to the Moran process.

We aim to enable the reader to apply drift theory in their own research to derive accessible proofs and to teach it as a simple tool for the analysis of random processes.
\end{abstract}

\section{Drift Theory}
\label{sec:introduction}

\input{introduction}

\section{The Classic Drift Theorems}
\label{sec:driftTheorems}

\input{driftTheorems}

\section{Coupon Collector, Winning Streaks and Vertex Cover}
\label{sec:simpleDrift}

\input{simpleDrift}

\section{Genereal Drift Corollaries}
\label{sec:generalCorollaries}

\input{generalCorollaries}

\section{Application of Corollaries}
\label{sec:corollaryApplications}

\input{corollaryApplications}

\section{Conclusions}
\label{sec:conclusions}

\input{conclusions}

\bibliography{../BeautifulDrift}

\newpage
\appendix
\section{Formal Explanations of our Simplifications}
\label{sec:appendix}

\input{appendix}

\end{document}

%% file: Commands.tex
\let\originalleft\left
\let\originalright\right
\renewcommand{\left}{\mathopen{}\mathclose\bgroup\originalleft}
\renewcommand{\right}{\aftergroup\egroup\originalright}

\makeatletter
\def\IfEmptyTF#1{%
	\if\relax\detokenize{#1}\relax
	\expandafter\@firstoftwo
	\else
	\expandafter\@secondoftwo
	\fi}
\makeatother

\DeclareDocumentCommand{\mathOrText}{m}
{%
    \ensuremath{#1}\xspace%
}

\usepackage{dsfont}
\renewcommand*{\Re}{\mathOrText{\mathds{R}}}
\newcommand*{\Na}{\mathOrText{\mathds{N}}}

\DeclareDocumentCommand{\functionTemplate}{m m m m o}
{%
	\IfNoValueTF{#5}%
	{%
		\mathOrText{#1\left#2{#4}\right#3}%
	}%
	{%
		\mathOrText{#1#5#2{#4}#5#3}%
	}%
}
\DeclareDocumentCommand{\probabilistcFunction}{m m O{} o}
{%
	\functionTemplate{#1}{[}{]}{#2\IfEmptyTF{#3}{}{\,\IfNoValueTF{#4}{\left}{#4}\vert\,\vphantom{#2}{#3}\IfNoValueTF{#4}{\right.}{}}}[#4]%
}

\DeclareDocumentCommand{\Pr}{m O{} o}
{%
	\probabilistcFunction{\mathrm{Pr}}{#1}[#2][#3]%
}
\DeclareDocumentCommand{\E}{m O{} o}
{%
	\probabilistcFunction{\mathrm{E}}{#1}[#2][#3]%
}
\DeclareDocumentCommand{\Var}{m O{} o}
{%
	\probabilistcFunction{\mathrm{Var}}{#1}[#2][#3]%
}

\DeclareDocumentCommand{\bigO}{m o}
{%
	\functionTemplate{\mathrm{O}}{(}{)}{#1}[#2]%
}
\DeclareDocumentCommand{\smallO}{m o}
{%
	\functionTemplate{\mathrm{o}}{(}{)}{#1}[#2]%
}
\DeclareDocumentCommand{\bigTheta}{m o}
{%
	\functionTemplate{\Theta}{(}{)}{#1}[#2]%
}
\DeclareDocumentCommand{\bigOmega}{m o}
{%
	\functionTemplate{\Omega}{(}{)}{#1}[#2]%
}
\DeclareDocumentCommand{\smallOmega}{m o}
{%
	\functionTemplate{\upomega}{(}{)}{#1}[#2]%
}
\DeclareDocumentCommand{\eulerE}{o}
{%
    \mathOrText{\mathrm{e}\IfNoValueTF{#1}{}{^{#1}}}%
}
\DeclareDocumentCommand{\poly}{m o}
{%
	\functionTemplate{\mathrm{poly}}{(}{)}{#1}[#2]%
}

%% file: additional_commands.tex
\newcommand*{\pplus}{\mathOrText{p^{\scriptscriptstyle\leftarrow}}}
\newcommand*{\pminus}{\mathOrText{p^{\scriptscriptstyle\rightarrow}}}

\DeclareDocumentCommand{\martingale}{o}
{%
    \mathOrText{X\IfNoValueTF{#1}{}{_{#1}}}%
}
\newcommand*{\timePoint}{\mathOrText{t}}
\newcommand*{\firstHittingTime}{\mathOrText{T}}
\DeclareDocumentCommand{\transformedProcess}{o}
{%
    \mathOrText{Y\IfNoValueTF{#1}{}{_{#1}}}%
}

%% file: introduction.tex
%
%global TODO-list
%\begin{itemize}
%	\item we use the Oxford comma
%	\item we use American English
%	\item check for ``run time'', not ``running time'' or ``runtime'' or ``run-time''.
%	\item Check for hyphen in ``one-barrier'' and ``two-barrier''.
%\end{itemize}

Suppose that you win a million dollars in a lottery and that you start spending your winnings.
You observe that you spend on average $\$\,10\,\textrm{k}$ per day. How long will your lottery winnings last?
Intuitively, you would divide the million you won by $\$\,10\,\textrm{k}$ and estimate that your winnings would last for $100$ days.
But that feels like confusing a random process with a deterministic one.
Well, yes, but the good news is: There is a theorem that tells us that 100 days is the mathematically precise answer, even when the process is randomized.
Even better, if you gain money on some days (say, by playing in a casino) but still, in expectation, your balance goes down by $\$\,10\,\textrm{k}$ per day, the conclusion still holds.
There can even be dependencies between the earnings/spendings of different days.
The theorem showing that this is the case is called the \emph{additive drift theorem} (see \Cref{thm:addDrift}).
The term \emph{drift} refers to the expected change of the stochastic process, and the term \emph{additive} refers to the requirement that two successive values of the process differ, in expectation, by an additive constant.

A similar setting to that of the process described above is the well-known \emph{coupon collector} process.
Suppose there are $n \in \Na_{\geq 1}$ collectible kinds of coupons, of which you would like a complete set.
Each iteration $t \in \Na$, you receive a uniformly random kind of coupon.
How long does it take until you have a complete set?
Note that if you still miss $X_t$ coupons from the complete collection after $t$ iterations, you have a chance of $X_t/n$ of getting a new one.
Thus, the expected gain is $X_t/n$.
This is a multiplicative expected progress and the \emph{multiplicative drift theorem} (see \Cref{thm:multiDrift}) gives an upper bound of $\bigO{n \log n}$.
Again, this theorem also holds when there is a possibility of losing coupons, and it even gives a concentration bound.
This brings us to the following description of drift theory:

\begin{center}
\begin{tikzpicture}
\node[rectangle,rounded corners, draw=black!50,fill=black!8, inner sep=4mm] at (0,0) {\parbox{10cm}{\textbf{Drift theory} is a collection of theorems to turn iteration-wise expected gains into expected first-hitting times.}};
\end{tikzpicture}
\end{center}

The first drift theorem, the \emph{additive} drift theorem, was introduced by He and Yao~\cite{DBLP:journals/ai/HeY01}, based on an intricate theorem by Hajek~\cite{hajek1982hitting}.
He and Yao applied their theorem in the context of analyzing randomized search heuristics (RSHs), such as evolutionary algorithms (EAs), which work by the principle of variation (mutating solutions by random changes) and selection (accepting improvements and rejecting worsenings). Drift theory gained a lot of traction in the EA theory community after the \emph{multiplicative} drift theorem was introduced by Doerr, Johannsen, and Winzen~\cite{DBLP:conf/gecco/DoerrJW10}.
Their proof used additive drift, but a proof not relying on Hajek's result was given shortly after by Doerr and Goldberg~\cite{DBLP:conf/ppsn/DoerrG10a}.
Since then, drift theory has been the dominant method for formally analyzing RSHs, easing their analysis.
For example, the main result of Droste~\cite{DBLP:conf/gecco/Droste04} on noisy optimization, spanning an entire paper, was reproven with drift in a more generalized fashion by Giessen and Kötzing~\cite{DBLP:journals/algorithmica/GiessenK16} on a single page.
There are drift theorems for analyzing processes with neither additive nor multiplicative drift (\emph{variable} drift), processes with drift in the wrong direction (\emph{negative} drift), as well as concentration bounds.
Lengler~\cite{Lengler20DriftBookChapter} provides an extensive overview, including applications to the analysis of RSHs.
Kötzing and Krejca~\cite{KoetzingK19GeneralDriftTheorems} provide self-contained proofs of the most commonly used drift theorems in the up-to-date most general version, via martingale theory.

In this paper, we discuss drift theorems formally in \Cref{sec:driftTheorems} and then show that drift theory simplifies proofs of many classical results in the analysis of randomized algorithms and stochastic processes.
Any researcher working in the field can probably make use of drift theorems since many questions concerning stochastic processes can be formulated as first-hitting times of a real-valued process. Frequently, drift theory is applied to processes that are not real-valued by mapping the states of the process to real values (using a potential function). However, in spite of the versatility of drift theorems, there are only very few results outside of the theory of RSHs applying drift theory \cite{bertschinger_et_al:LIPIcs:2020:12883,DBLP:journals/corr/abs-1804-02293,10.1145/3212734.3212788}.

In \Cref{sec:simpleDrift}, we discuss six theorems where we easily uncover a drift and straightforwardly apply drift theorems. We start with two versions of the coupon collector process mentioned earlier, where the second is a generalization that uses the same proof as the special case. Both proofs employ the multiplicative drift theorem. Also in this section, we show how the expected value of a geometric distribution can be found via drift, followed by an analysis of the waiting time for a winning streak of a certain length (this proof employs an elaborate potential function). Finally, we consider two simple randomized algorithms, one to find a 2-approximation for vertex cover by picking a random vertex of an uncovered edge, and one for sorting, swapping two randomly chosen elements if these elements are not in order.

In \Cref{sec:generalCorollaries}, we give two theorems applicable for stochastic processes that do not exhibit drift (i.e., have an expected change of $0$), such as martingales, where a direct drift analysis is not applicable.
We show how to transform such processes to derive first-hitting time bounds. Also in this section, we cite a very precise theorem on processes in finite state spaces.

We apply these theorems in \Cref{sec:corollaryApplications} to bound the first-hitting time of the gambler's ruin process, the classical process with expected change of~$0$, where each iteration sees either an increase or a decrease of $1$, each with probability $1/2$. Further, we bound the expected run time of two randomized algorithms: one for finding colorings of a graph and one for finding satisfying assignments of a 2-SAT formula. Last, we consider a version of the \emph{Moran process}, a stochastic process that arises in biology and models the spread of genetic mutations in populations (see \Cref{sec:moran} for details). The analysis of the Moran processes is one of the rare cases where drift theory was used previously outside of the analysis of RSHs~\cite{DGMRSS14}.

We conclude our paper in \Cref{sec:conclusions}. Overall, we consider drift theory very intuitive and widely applicable. With this paper, we want to put it as an option into the toolbox of the reader, and to show its accessibility, which also makes it a great tool to teach to students.

%% file: driftTheorems.tex
% The random variable denoting a random process
\DeclareDocumentCommand{\randomProcess}{o}
{%
    \mathOrText{X\IfNoValueTF{#1}{}{_{#1}}}%
}

We state the three most commonly used drift theorems in order of increasing complexity.
Since the theorems are rather technical and we aim for simplicity, we make certain simplifications throughout this article, which we briefly explain in the following.
For a more formal explanation of our simplifications, please also refer to \Cref{sec:appendix}.

\subparagraph{Simplifications}
At its core, drift theory is a reformulation of the \emph{optional-stopping theorem} for martingales.
Martingales are a very powerful tool for analyzing random processes, but they require knowledge about advanced concepts from probability theory, such as filtrations, stopping times, or expected values conditional on a $\upsigma$-algebra.
In this article, we only consider (discrete) Markov chains (that is, processes where the transition \emph{exclusively} depends on the current state and time point), define specific first-hitting times, and formulate the drift condition using expected values conditional on an event, not on a $\upsigma$-algebra.
None of these assumptions are necessary but simplify the statement of the theorems.
More general versions of drift theorems are stated by Kötzing and Krejca~\cite{KoetzingK19GeneralDriftTheorems}.

\subparagraph{Drift Theorems}
The first drift theorem is the additive drift theorem, which requires a uniform bound for the expected change of a process.

\begin{theorem}[\textrm{Additive drift~\cite{DBLP:journals/ai/HeY01, DBLP:journals/nc/HeY04}}]
    \label{thm:addDrift}
    Let $(\randomProcess[\timePoint])_{\timePoint \in \Na}$ be an integrable Markov chain over $S \subset \Re$, and let $T = \inf\{\timePoint \in \Na \mid \randomProcess[\timePoint] \leq 0\}$.
    \begin{enumerate}
        \item\label{item:addDriftUpperBound} Assume that there is a $\delta \in \Re_{> 0}$ such that, for all $s \in S \cap \Re_{> 0}$ and all $\timePoint < \firstHittingTime$, it holds that

                $\E{\randomProcess[\timePoint] - \randomProcess[\timePoint + 1]}[\randomProcess[\timePoint] = s] \geq \delta\ ,$ and that,
            for all $\timePoint < \firstHittingTime$, we have $\randomProcess[\timePoint] \geq 0\ .$
            Then
            \[
                \E{\firstHittingTime} \leq \frac{\E{\randomProcess[0]}}{\delta}\ .
            \]

        \item\label{item:addDriftLowerBound} Assume that there is a $\delta \in \Re_{> 0}$ such that, for all $s \in S \cap \Re_{> 0}$ and all $\timePoint  < \firstHittingTime$, it holds that $\E{\randomProcess[\timePoint] - \randomProcess[\timePoint + 1]}[\randomProcess[\timePoint] = s] \leq \delta\ ,$ and that
            there is a $c \in \Re_{\geq 0}$ such that, for all $s \in S \cap \Re_{> 0}$ and all $\timePoint  < \firstHittingTime$, it holds that $\E{|\randomProcess[\timePoint] - \randomProcess[\timePoint + 1]|}[\randomProcess[\timePoint] = s][\big] \leq c\ .$
            Then
            \[
                \E{\firstHittingTime} \geq \frac{\E{\randomProcess[0]}}{\delta}\ .
            \]
    \end{enumerate}
\end{theorem}

Note that, for case~\ref{item:addDriftUpperBound}, we cannot allow the process to assume smaller values than the target~$0$. This is demonstrated by the process $(\randomProcess[\timePoint])_{\timePoint \in \Na}$ with $X_0 = 1$ and, for all $t$, with probability $1-1/n$, $X_{t+1} = X_t$ and otherwise $X_{t+1} = -n+1$; even the expected time until $X_t \leq 0$ is not correctly bounded.

For case~\ref{item:addDriftLowerBound}, the expected step size of the process must be bounded, as demonstrated by the process $(\randomProcess[\timePoint])_{\timePoint \in \Na}$ with $X_0 = 1$ and, for all $t$, with probability $1/2$, $X_{t+1} = 0$ and otherwise $X_{t+1} = 2X_t - 2\delta$.

%Note that, for computing the drift, it is frequently useful to find a bound conditional on the current state and derive $\E{\randomProcess[\timePoint] - \randomProcess[\timePoint + 1]}[\randomProcess[\timePoint]] \geq \delta$. Then one can apply the law of total expectation to get
%$$
%\E{\randomProcess[\timePoint] - \randomProcess[\timePoint + 1]} = \E{\E{\randomProcess[\timePoint] - \randomProcess[\timePoint + 1]}[\randomProcess[\timePoint]]}[][\big] \geq \E{\delta} = \delta\ .
%$$
The following drift theorems are both motivated by the observation that various processes exhibit a strength of the drift depending on the current value of the process.
The multiplicative drift theorem covers the case where the drift is proportional to the current value of the process.

\begin{theorem}[\textrm{Multiplicative drift~\cite{DBLP:conf/gecco/DoerrJW10} with tail bounds~\cite{DBLP:conf/ppsn/DoerrG10a}}]
    \label{thm:multiDrift}
    Let $(\randomProcess[\timePoint])_{\timePoint \in \Na}$ be an integrable Markov chain over $\{0, 1\} \cup S$, where $S \subset \Re_{> 1}$, and let $\firstHittingTime = \inf\{\timePoint \in \Na \mid \randomProcess[\timePoint] = 0\}$.

    Assume that there is a $\delta \in \Re_{> 0}$ such that, for all $s \in S \cup \{1\}$ and all $\timePoint < \firstHittingTime$, it holds that
  $\E{\randomProcess[\timePoint] - \randomProcess[\timePoint + 1]}[\randomProcess[\timePoint] = s] \geq \delta s\ .$
    Then
    \[
        \E{\firstHittingTime} \leq \frac{1 + \ln \E{\randomProcess[0]}}{\delta}\ .
    \]
    Further, for all $k > 0$ and $s \in S \cup \{1\}$, it holds that
    \[
        \Pr{\firstHittingTime > \frac{k + \ln s}{\delta}}[\randomProcess[0] = s] \leq \mathrm{e}^{-k}\,.
    \]
\end{theorem}

A more general version of both theorems above is the variable drift theorem, which allows for any \emph{monotone} dependency of the drift on the current state (meaning that a larger distance to the target has to imply a larger drift).
\begin{theorem}[\textrm{Variable drift~\cite{DBLP:journals/ijicc/MitavskiyRC09, Johannsen10}}]
    \label{thm:variDrift}
    Let $(\randomProcess[\timePoint])_{\timePoint \in \Na}$ be an integrable Markov chain over $\{0, 1\} \cup S$, where $S \subset \Re_{> 1}$, and let $\firstHittingTime = \inf\{\timePoint \in \Na \mid \randomProcess[\timePoint] = 0\}$.

    % The variable drift function
    \DeclareDocumentCommand{\variableDriftFunction}{o}
    {%
        \mathOrText{h\IfNoValueTF{#1}{}{\left({#1}\right)}}%
    }
    If there exists a monotonically increasing function $\variableDriftFunction\colon \Re^+ \to \Re_{\geq 0}$ such that, for all $s \in S \cup \{1\}$ and all $\timePoint < \firstHittingTime$, it holds that
    $\E{\randomProcess[\timePoint] - \randomProcess[\timePoint + 1]}[\randomProcess[\timePoint] = s] \geq \variableDriftFunction[s]\ ,$
    then
    \[
        \E{\firstHittingTime} \leq \frac{1}{\variableDriftFunction[1]} + \int_{1}^{\E{\randomProcess[0]}}\frac{1}{\variableDriftFunction[x]}\,\mathrm{d}x\ .
    \]
\end{theorem}

These three drift theorems already cover a lot of applications.
In fact, in this article, we only need the additive- and the multiplicative drift theorem; we stated the variable drift theorem solely as an illustration.
Further drift theorem variants exist, for example, for lower bounds on multiplicative drift~\cite{Witt13MultiLowerBound} or on variable drift~\cite{Lehre2014}.
Moreover, there are concentration bounds for hitting times under additive drift~\cite{DBLP:journals/algorithmica/Kotzing16} and occupation probabilities under additive drift~\cite{KotzingLW15}.
When the drift goes away from the target, we speak of \emph{negative drift}.
The negative drift theorem \cite{OlivetoW11,OlivetoW12} gives an exponential lower bound in this setting.

We note that, in our applications of the drift theorems in the following, we do not show that the random process under consideration is an integrable Markov chain, since this is easily seen from the context that they are defined in.

%% file: simpleDrift.tex
We start this section with the classic coupon collector process, given in the introduction.

\begin{theorem}[Coupon Collector]
Suppose we want to collect at least one of each kind of $n \in \Na_{\geq 1}$ coupons. Each round, we are given one coupon chosen uniformly at random from the~$n$ kinds. Then, in expectation, we have to collect for at most $n(1+ \ln n)$ iterations. Furthermore, for all $k \in \Re_{> 0}$, overshooting this time by $kn$ has a probability of at most $\mathrm{e}^{-(k+1)}$\!.
\end{theorem}
\begin{proof}
Let $X_t$ be the number of coupons missing after $t$ iterations and $s$ an outcome of \randomProcess[\timePoint] before we collected all coupons. The probability of making progress (of $1$) with coupon $t+1$ is $X_t/n$. Thus, $\E{X_t - X_{t+1}}[X_t = s] = s/n$. An application of the multiplicative drift theorem (\Cref{thm:multiDrift}) gives the desired result.
\end{proof}

One can easily derive a lower bound on the expected first-hitting time with the same asymptotic growth with a lower-bounding multiplicative drift theorem~\cite{Witt13MultiLowerBound}. Using the analogous proof, one can directly analyze a \emph{generalized version} of the coupon collector process as follows.

\begin{theorem}[Generalized Coupon Collector]
Suppose we want to collect at least one of each kind of $n \in \Na_{\geq 1}$ coupons. For each kind of coupon and each round, we get this kind of coupon with probability at least $p \in (0, 1]$. Then, in expectation, we have to collect for at most $(1+ \ln n)/p$ iterations. Furthermore, for all $k \in \Re_{> 0}$, overshooting this time by $k/p$ has a probability of at most $\mathrm{e}^{-(k+1)}$.
\end{theorem}
\begin{proof}
Let $X_t$ be the number of coupons missing after $t$ iterations and $s$ an outcome of \randomProcess[\timePoint] before we collected all coupons. The expected progress is $\E{X_t - X_{t+1}}[X_t = s] \geq p s$, since the expected number of missing coupons that we get in the next iteration is $p s$. An application of the multiplicative drift theorem (\Cref{thm:multiDrift}) gives the desired result.
\end{proof}

Note that this generalized version does not make any assumptions on how many coupons we get per iteration, or whether these indicator random variables are in any way correlated.

We now turn to the well-known geometric distribution. The typical computation for its expectation involves modifying infinite sums. Using drift, the computation is rather simple.

\begin{theorem}[Geometric Distribution]
Suppose there is an iterative process where, in each iteration, a \emph{success} event happens with some probability $p \in (0, 1]$. Then the expected time until the first success event is~$1/p$.
\end{theorem}
\begin{proof}
For all $t \in \Na$, let $X_t$ be~$0$ if a success event has happened within the first~$t$ iterations, and $1$ otherwise. Then $\E{X_t - X_{t + 1}}[X_t = 1] = p$. Thus, the additive drift theorem (\Cref{thm:addDrift}) gives us an expected number of $1/p$ iterations until the first success event.
\end{proof}
Note that this proof trivially extends to having only a lower or upper bound on the probability of achieving success in any iteration (resulting in upper or lower bounds on the expected time until first success, respectively). Furthermore, in this case, the chance of success in a given iteration may depend on successes in preceding iterations, just as long as it is at least (respectively at most) $p$ regardless of the outcomes of previous iterations.

We now considers sequences of random bits and run lengths therein. In the proof we apply the additive drift theorem not going down towards $0$, but going up to a value of $k$. Since the additive drift is symmetrical, we use it in either direction equally.

\begin{theorem}[Winning Streaks]
Let $k \in \Na$ be given. Consider flipping a fair coin indefinitely. Then the expected number of iterations until the first time that \emph{heads} comes up $k$ times in a row is (exactly) $f(k) = 2^{k+1} - 2$.
\end{theorem}
\begin{proof}
	For all $t \in \Na$, let $R_t$ be the length of the current streak of heads after $t$ iterations ($R_t = 0$ if in iteration $t$ we got tails, as well as before any coin flip).
	Let $X_t = f(R_t)$ be our process for which we aim to show drift. Let $i \in \Na$ be given. If our current streak of heads is~$i$, then in the next iteration one of two things happens: either we lose all progress, falling to a potential of $f(0) = 0$, or we gain $f(i+1) - f(i)$. Each happens with probability $1/2$, so we have
	\begin{align*}
    	\E{X_{t+1} - X_t \mid X_t = f(i)}
    	 = \frac{1}{2}f(i+1)  - f(i)
    	 = 2^{i+2}/2 - 2/2 - (2^{i+1} - 2)
    	 = 1.
	\end{align*}
	Thus, using the additive drift theorem (\Cref{thm:addDrift}), going up instead of down, we get an expected number of iterations of $f(k) = 2^{k+1}-2$ to reach a streak of $k$ heads.
\end{proof}

	Note that the potential function in the last proof, as in many places where potential functions are used, is not intuitive, so let us discuss where this potential function comes from. We decide we want to set up for additive drift, since the additive drift theorem gives both lower and upper bounds. Since any potential function that gives an additive drift can be normalized to give an additive drift of $1$, we search for a potential function that gives a drift of exactly $1$. From the two possible outcomes of the coin flipping process in each iteration, we now get the condition of $f(i+1)/2 - f(i) = 1$ for the potential $f$. In this case, this is a straightforward and easy to solve recurrence relation, so that with the (arbitrary) setting of $f(0) = 0$ we get the desired formula for $f$.

\subparagraph{More Complex Problems}
The following two examples consider processes that are not easily fully described by a single number.
In order to nonetheless apply drift theory, we use slightly more advanced versions of the drift theorems.
In these settings, the random process $(X_t)_{t \in \Na}$ and the states~$s$ are defined over \emph{some} space~$\mathcal{S}$.
In addition, we require a function $f\colon \mathcal{S} \to \Re$.
When exchanging all occurrences of~$X_t$ in a drift theorem that are \emph{not} inside a condition by $f(X_t)$, the statement remains true.
This allows us to condition on more complex processes.

Our next example is a randomized algorithm for finding, in expectation, a $2$-approximation of the classical vertex cover problem. For an undirected graph $(V, E)$, a subset $C \subseteq V$ such that, for all $\{u, v\} \in E$, $u$ or $v$ is in~$C$ is called a \emph{vertex cover}. By the additive drift theorem (\Cref{thm:addDrift}), we easily bound the expected size of the vertex cover that the algorithm constructs.

\begin{theorem}[Vertex Cover Approximation]
Given an undirected graph, iteratively choose an uncovered edge and add uniformly at random an endpoint to the cover. Then, in expectation, the resulting cover is a $2$-approximation of an optimal vertex cover of the given graph.
\end{theorem}
\begin{proof}
Let a graph $G$ be given. Furthermore, fix a minimum vertex cover $C$. For all $t$, let $D_t$ be the set of vertices chosen by the algorithm after $t$ iterations. Let $X_t$ be $0$ if $D_t$ is a vertex cover, and otherwise let $X_t$ be the number of vertices of $C$ that are not in $D_t$. Clearly, the algorithm terminates exactly when $\randomProcess[\timePoint] = 0$. Furthermore, in each step, the algorithm selects a vertex from $C$ with probability at least $1/2$, since, for every edge of $G$, at least one of the endpoints is in $C$. Let~$s$ denote a set of vertices. We get $\E{\randomProcess[\timePoint] - \randomProcess[\timePoint + 1]}[D_{\timePoint} = s] \geq \frac{1}{2}$. Hence, using the additive drift theorem (\Cref{thm:addDrift}), we get that the algorithm terminates in expectation after choosing $2|C|$ vertices.
\end{proof}

Our last example is a simple randomized sorting algorithm. This and similar sorting algorithms were considered by Scharnow, Tinnefeld, and Wegener~\cite{STW05}. The analysis via the multiplicative drift theorem is short, easy, and intuitive.

\newcommand*{\arrayName}{\mathOrText{A}}
\begin{theorem}[Random Sorting]
Consider the sorting algorithm which, given an input array \arrayName of length $n \in \Na_{\geq 1}$, iteratively chooses two different
positions  of the array uniformly at random and swaps them if and only if they are out order. Then the algorithm obtains a sorted array after
$\Theta(n^2 \log n)$ iterations in expectation.
\end{theorem}
\begin{proof}
For all $i, j \in [1, n] \cap \Na$ with $i < j$, an ordered pair $(i,j)$ is called an \emph{inversion}
if and only if $\arrayName[i] > \arrayName[j]$. Note that the maximum number of inversions is $\binom{n}{2}$. Let $X_t$ be the number of inversions after $t \in \Na$ iterations, and let $A_t$ denote the array during that iteration. If the algorithm chooses a pair which is not an inversion, nothing changes. If the algorithm chooses an inversion $(i,j)$, then this inversion is removed; for any other inversion, only indices $k \in [1, n] \cap \Na$ with $i < k < j$ are relevant.
If $\arrayName_{\timePoint}[k] < \arrayName_{\timePoint}[j]$ ($< \arrayName_{\timePoint}[i]$), then $(i,k)$ is an inversion before and after the swap, while $(k,j)$ is neither an
inversion before nor after the swap; similarly for $\arrayName_{\timePoint}[k] > \arrayName_{\timePoint}[i]$ ($> \arrayName_{\timePoint}[j]$). Finally, if $\arrayName_{\timePoint}[j] < \arrayName_{\timePoint}[k] < \arrayName_{\timePoint}[i]$, then
$(i,k)$ and $(k,j)$ are inversions before the swap but are not afterwards. Overall, this shows that the number of
inversions goes down by at least $1$ whenever the algorithm chooses an inversion for swapping.
Let $\timePoint$ be such that $\arrayName_{\timePoint}$ is not sorted and let $a$ denote an outcome of $\arrayName_{\timePoint}$. Let $s$ be the number of inversions of $\arrayName_{\timePoint}$ (which means $X_t = s$).
%Let $s$ denote an outcome of \randomProcess[\timePoint] before~$\arrayName_{\timePoint}$ is sorted, and let~$a$ denote an outcome of~$\arrayName_{\timePoint}$.
Since the probability of the algorithm choosing an inversion is $s/\binom{n}{2}$, we get $\E{X_t - X_{t+1}}[\arrayName_{\timePoint} = a] \geq s/\binom{n}{2}$. An application of the multiplicative drift theorem (\Cref{thm:multiDrift}) gives the desired upper bound.

Regarding the lower bound, consider the array $\arrayName$ which is almost sorted but the first and second element are swapped,
the third and fourth, and so on. Then the algorithm effectively performs a coupon collector process on $n/2$ coupons,
where each has a probability of $1/\binom{n}{2}$ to be collected. This takes an expected time of $\Omega(n^2 \log n)$.
\end{proof}

%% file: generalCorollaries.tex
In this section, we state three corollaries to the drift theorems in specific application domains. The first domain is \emph{drift without drift}, using drift theorems to unbiased random walks; \Cref{thm:variance-two-barrier-hitting-time,thm:variance-one-barrier-hitting-time} give the situation for random walks with two and one barriers, respectively. Then, in \Cref{thm:generalFiniteStateSpace}, we show an application to finite state spaces. All these theorems can be applied generally, which we exemplify in \Cref{sec:corollaryApplications}.
For these theorems, we note that our proofs use the same advanced drift theorems that we mention in \Cref{sec:simpleDrift} when discussing more complex processes, that is, we condition on a process that is different than the one whose drift we analyze.

\begin{theorem}[Unbiased Random Walk on the Line]\label{thm:variance-two-barrier-hitting-time}
    Let $n \in \Na$, let $(X_t)_{t \in \Na}$ be an integrable Markov chain over $[0, n] \cap \Na$, and let $T = \inf\{t \in \Na \mid X_t \in \{0, n\}\}$.
    Suppose that there is a $\delta \in \Re_{> 0}$ such that, for all $s \in S \setminus \{0, n\}$ and all $t < T$, it holds that
    \[
        \E{X_{t+1} - X_t}[X_t = s] = 0 \text{ and } \Var{X_{t+1} - X_t}[X_t = s] = \delta\ .
    \]
    Then $\E{T} = \frac{\E{\martingale[0](n - \martingale[0])}}{\delta}$.
\end{theorem}
\begin{proof}
    We consider the process $Y_t = X_t(n-X_t)$.
    Note that~$T$ is the first time $t \in \Na$ such that $Y_t = 0$. Further note that, for all $s \in S \setminus \{0, n\}$ and all $t < T$, it holds that $\E{X_{t + 1}}[X_t = s] = s$. For all $s \in S \setminus \{0, n\}$, we have
    \begin{align*}
        \E{Y_t - Y_{t+1} \mid X_t = s} &= \E{X_{t+1}^2-X_{t}^2 \mid X_t = s} - n\E{X_{t+1}-X_{t} \mid X_t = s}\\
        &= \E{X_{t+1}^2}[X_t = s] - s^2 = \Var{X_{t+1}}[X_t = s]\\
        &= \Var{X_{t+1} - s}[X_t = s] = \Var{X_{t+1} - X_t}[X_t = s] = \delta\ .
    \end{align*}
    Thus, we have a drift of $\delta$ towards $0$. Since $Y_0 = X_0(n - X_0)$, the theorem follows from an application of the additive drift theorem (\Cref{thm:addDrift}).
\end{proof}
Since the proof is based on the additive drift theorem, a lower bound of $\delta$ on the variance is enough for an upper bound on the expected first-hitting time and vice versa.

\begin{theorem}\label{thm:variance-one-barrier-hitting-time}
    Let $n \in \Na$, let $(X_t)_{t \in \Na}$ be an integrable Markov chain over $[0, n] \cap \Na$, and let $T = \inf\{t \in \Na \mid X_t = n\}$.
    Suppose that there is a $\delta \in \Re_{> 0}$ such that, for all $s \in S \setminus \{n\}$ and all $t < T$, it holds that
    \[
        \E{X_{t+1} - X_t}[X_t = s] \geq 0 \text{ and } \Var{X_{t+1} - X_t}[X_t = s] \geq \delta\ .
    \]
    Then $\E{T} \leq \frac{n^2 - \E{\martingale[0]^2}}{\delta}$.
\end{theorem}
\begin{proof}
    We aim to apply the additive drift theorem (\Cref{thm:addDrift}) to the process $Y_t = n^2 - X_t^2$.
    Note that~$T$ is the first time such that $Y_t = 0$. Further note that, for all $s \in S \setminus \{n\}$ and all $t < T$, it holds that $\E{X_{t + 1}^2}[X_t = s]^2 \geq s^2$. For all $s \in S \setminus \{n\}$, we have
    \begin{align*}
        \E{Y_t - Y_{t+1} \mid X_t = s} &= \E{X_{t+1}^2 - X_{t}^2}[X_t = s] = \E{X_{t+1}^2}[X_t = s] - s^2\\
        &\geq \E{X_{t+1}^2}[X_t = s] - \E{X_{t + 1}}[X_t = s]^2 = \Var{X_{t+1}}[X_t = s]\\
                &= \Var{X_{t+1} - s}[X_t = s] = \Var{X_{t+1} - X_t}[X_t = s] \geq \delta\ .
    \end{align*}
    Thus, we have a drift of at least $\delta$ towards $0$. Since $Y_0 = n^2 - \martingale[0]^2$, the theorem follows.
\end{proof}

Note that in neither of the two preceding theorems is the process allowed to overshoot the target.

We note that \Cref{thm:variance-one-barrier-hitting-time} is tight for the fair random walk $(X_t)_{t \in \Na}$ over $[0, n] \cap \Na$ where the state~$0$ is reflective, that is, for all $t \in \Na$, it holds that $\Pr{X_{t + 1} = 1}[X_t = 0] = 1$, and the state~$n$ is absorbing.
This is seen by transforming~$X$ into the fair random walk $(Y_t)_{t \in \Na}$ over $[0, 2n] \cap \Na$, where the states~$0$ and~$2n$ are both absorbing, such that, for all $t \in \Na$, it holds that $X_t = |Y_t - n|$.
Informally, we mirror~$X$ at~$0$ and then shift it by~$n$.
Whenever this new process is at~$n$, it goes to either~$n - 1$ or~$n + 1$, each with probability~$1/2$, which results exactly in~$Y$.
Note that $T = \inf\{t \in \Na \mid Y_t \in \{0, 2n\}\} = \inf\{t \in \Na \mid X_t = n\}$.
Applying \Cref{thm:variance-two-barrier-hitting-time} to~$Y$ yields $\E{T} = \E{Y_0(2n - Y_0)}$.
Since $X_0 \leq n$, it holds that $Y_0 = n - X_0$.
Substituting this back into the equation for~\E{T} yields $\E{T} = \E{(n - X_0)(n + X_0)} = n^2 - \E{X_0^2}$, which is exactly the bound of \Cref{thm:variance-one-barrier-hitting-time}.

The following theorem by Kötzing and Krejca~\cite[Theorem 3]{KoetzingK18FiniteStateSpaces} allows for a finer analysis of processes over finite state spaces, as the transition probabilities in each state may use different bounds.
It is applicable to processes that may make arbitrarily large jumps toward the target state (not overshooting it) but that do not move away from the target state by more than a single state at a time.
The authors also provide a version that bounds the expected first-hitting time from below.

\begin{theorem}[Finite Space Drift {\cite{KoetzingK18FiniteStateSpaces}}]\label{thm:generalFiniteStateSpace}
    Let $n \in \Na$, let $(X_t)_{t \in \Na}$ be an integrable Markov chain over $[0, n] \cap \Na$, and let $\firstHittingTime = \inf\{t \in \Na \mid X_t = 0\}$.
    Suppose there are two functions $\pplus\colon \{1, \ldots, n\} \to [0, 1]$ and $\pminus\colon \{0, \ldots, n-1\} \to [0, 1]$ such that, for all $s \in \{1, \ldots, n\}$ and all $\timePoint < \firstHittingTime$, it holds that
    \begin{itemize}
   		\item $\pplus(s) > 0$,

   		\item $\Pr{\randomProcess[\timePoint] - \randomProcess[\timePoint + 1] \geq 1}[\randomProcess[\timePoint] = s] \geq \pplus(s)$,

   		\item $\Pr{\randomProcess[\timePoint] - \randomProcess[\timePoint + 1] = -1}[\randomProcess[\timePoint] = s] \leq \pminus(s)$ (for $s \neq n$), and

   		\item $\Pr{\randomProcess[\timePoint] - \randomProcess[\timePoint + 1] < -1}[\randomProcess[\timePoint] = s] = 0$ (for $s \neq n$).
    \end{itemize}
    Then, for all $s_0 \in [0, n] \cap \Na$, it holds that
    \[
   		\E{\firstHittingTime}[\randomProcess[0] = s_0] \leq \sum_{s = 1}^{s_0}\sum_{i = s}^{n}\frac{1}{\pplus(i)}\prod_{j = s}^{i - 1}\frac{\pminus(j)}{\pplus(j)}\ .
    \]
\end{theorem}

We note that if the process changes by at most one and the bounds in the second and third condition are tight, then the bound holds with equality, showing the tightness of the bound for an important subcase.

%% file: corollaryApplications.tex
In this section we see several domains in which we apply our corollaries. The gambler's ruin in Section~\ref{sec:drunkardsWalk} is the most straightforward application
% of the unbiased random walk on a line (
\Cref{thm:variance-two-barrier-hitting-time}. A more intricate application is given in Section~\ref{sec:recolour}, where it is used to bound the expected run time for an algorithm to find a certain coloring of a graph. Finally, in Section~\ref{sec:moran}, we discuss a process modeling the spread of mutations in populations, where one case can be solved with \Cref{thm:variance-two-barrier-hitting-time}.

Regarding \Cref{thm:variance-one-barrier-hitting-time}, in Section~\ref{sec:twosat} we use it to derive an upper bound on the time for an algorithm to find a satisfying assignment for a 2-SAT formula.

We mentioned Section~\ref{sec:moran} earlier, where an easy case can be solved with the unbiased random walk on a line. The other two cases are solved by applying \Cref{thm:generalFiniteStateSpace}.

\subsection{Gambler's Ruin}\label{sec:drunkardsWalk}

The gambler's ruin is a random walk on a line, starting at $n$, going either one step left of one step right, each with probability $1/2$, modeling winning or losing a fair coin toss to either win or lose a coin. The question of how long it takes to either be broke ($0$ coins left) or double the starting number of coins is the simplest setting of an unbiased random walk. This process also goes by many other names, such as \emph{drunkard's walk}, \emph{random walk on a line}, or \emph{one-dimensional random walk}.

\begin{theorem}\label{thm:drunkardsWalk}
    Suppose we start with $n \in \Na$ coins and, in each iteration, uniformly at random either gain a coin or lose a coin. Then, after an expected number of exactly $n^2$ iterations, we are either broke or have reached a total of $2n$ coins.
\end{theorem}
\begin{proof}
    We apply \Cref{thm:variance-two-barrier-hitting-time} with a variance of exactly $1$.
\end{proof}

\subsection{Graph Coloring}\label{sec:recolour}

McDiarmid~\cite{McD93} studies the following simple randomized algorithm called \textsc{Recolour}, for coloring a given undirected graph~$G$ with two colors such that it contains no monochromatic triangle:
\textsc{Recolour} starts with an arbitrary $2$-coloring of~$G$. At every step, it checks whether
the current coloring has a monochromatic triangle. If so, \textsc{Recolour} changes the color of one of the vertices of
this triangle uniformly at random. Otherwise, the $2$-coloring has no monochromatic triangles and it is the output of
\textsc{Recolour}.

McDiarmid shows that when \textsc{Recolour} is applied to a 3-colorable graph~$G$, it returns a 2-coloring of~$G$ with no monochromatic triangle in expected time $\bigO{n^4}$. His analysis shows that the expected run time of the algorithm is bounded above by the expected hitting time of a random walk on the line with two absorbing states\footnote{Which is exactly the setting of \Cref{thm:variance-two-barrier-hitting-time}.}\!. This analysis in turn relies on previous results on one-dimensional random walks, which usually require lengthy calculations.

We present a simple and self-contained proof of the $\bigO{n^4}$ expected run time of the \textsc{Recolour} algorithm for finding
a 2-coloring with no monochromatic triangles on 3-colorable graphs. Our proof follows the proof of McDiarmid~\cite{McD93} to
reduce the problem to an unbiased random walk on the line and then uses \Cref{thm:variance-two-barrier-hitting-time}. A similar
analysis can be used to derive an upper bound on the run time of \textsc{Recolour} on hypergraph colorings.

\begin{theorem}[McDiarmid '93]
 The expected run time of \textsc{Recolour} on a 3-colorable graph with $n \in \Na^+$ vertices is $\bigO{n^4}$.
\end{theorem}
\begin{proof}
Let $G=(V,E)$ be a 3-colorable graph, and let $\chi\colon V\rightarrow \{1, 2, 3\}$
be a 3-coloring of $G$.
Let $U=\big\{v\in V\mid \chi(v)\in\{1, 2\}\big\}$  be the set of all vertices which are colored with colors $1$ and~$2$. Note that any 2-coloring of~$G$ that agrees with~$\chi$ on the vertices from~$U$ is a 2-coloring of~$G$ with no monochromatic triangles.
Thus, the run time of \textsc{Recolour} is bounded from above by the expected time that \textsc{Recolour} takes to find such a coloring.

Let $\chi_t$ be the 2-coloring found by \textsc{Recolour} at time $t \in \Na$. Let $Y_t$ be the number of vertices~$u\in U$ such
that $\chi_t(u)=\chi(u)$. The algorithm terminates when $Y_t \in \{0,|U|\}$, since agreeing on all vertices of $U$ is a coloring without monochromatic triangles, but disagreeing on all vertices from $U$ is also such a valid coloring, since the use of the colors is symmetric.

Let $s \in [1, |U| - 1] \cap \Na$ denote an outcome of $Y_t$ before the algorithm terminates. We then have that $\Pr{Y_{t+1}=Y_t+1}[Y_t = s]=1/3$, as, for every monochromatic triangle, there is exactly one
vertex in $u \in U$ with $\chi_t(u) \neq \chi(u)$. Similarly, $\Pr{Y_{t+1}=Y_t-1}[Y_t = s]=1/3$. Thus, $Y_t$ is an unbiased random walk on the line with first-hitting time $\firstHittingTime =
\inf\!\big\{\timePoint \in \Na \mid Y_\timePoint\in\{0,|U|\}\big\}$. Applying
\Cref{thm:variance-two-barrier-hitting-time} with variance $2/3$, we get
\[
\E{T}=\frac{3\E{Y_0(|U| - Y_0)}}{2}\leq \frac{3n^2}{8}\ .
\]
At each step, the algorithm requires $\bigO{n^2}$ time to transition to the next coloring, 
% (since there are only $\bigO{n^2}$ triangles to consider, as otherwise the graph would not be 3-colorable), 
which concludes the proof.
\end{proof}

The analysis of the \textsc{Recolour} algorithm for finding 2-colorings with no monochromatic triangles appears as an
exercise in~\cite[Exercise~7.10]{MU05}.

\subsection{Random 2-SAT}\label{sec:twosat}

Papadimitriou~\cite{P91} studies the following simple randomized algorithm that returns a satisfying assignment
of a satisfiable 2-SAT formula~$\phi$ with $n$ variables within $\bigO{n^4}$ time in expectation: the algorithm starts
with a random assignment of the variables of $\phi$. At every step, the algorithm checks whether there is an unsatisfied
clause for this assignment. If so, the algorithm changes the assignment of one of the variables of this assignment
uniformly at random. Otherwise, the assignment is satisfying and it is the output of the algorithm.

The analysis given is similar to the \textsc{Recolour} algorithm, and it also relies on the previous results on one-dimensional random walks. An extensive analysis of this algorithm appears in~\cite[Section~7.1.1]{MU05}. Here, we present a simpler proof that uses \Cref{thm:variance-one-barrier-hitting-time}.
\begin{theorem}
 The randomized 2-SAT algorithm, when run on a satisfiable 2-SAT formula over $n \in \Na^+$ variables, terminates in $\bigO{n^4}$ time in expectation.
\end{theorem}
\begin{proof}
Let $\phi$ be a satisfiable 2-SAT formula and let $a$ be a satisfying assignment. At each time step $t \in \Na^+$, the randomized
2-SAT algorithm finds a (not necessarily satisfying) assignment~$a_t$. Let $X_t$ be the random variable denoting the
number of variables that have the same truth assignment in both $a$ and $a_t$. Let $T$ be the first time the algorithm
reaches a satisfying assignment for $\phi$. Assume that a clause $x\vee y$ is not satisfied by $a_t$.
Since~$a$ is a satisfying assignment, $a$ and $a_t$ differ in the assignment of at least one of the variables in
this clause. Let $s \in [0, n - 1] \cap \Na$ denote an outcome of \randomProcess[\timePoint] before a satisfying assignment is found. Thus, $\Pr{X_{t+1}=X_t+1}[X_t = s]\geq1/2$ and $\Pr{X_{t+1}=X_t-1}[X_t = s]\leq1/2$. When $a_t=a$, the algorithm
terminates.
By \Cref{thm:variance-one-barrier-hitting-time} with variance bound $1$, we have $\E{T} \leq n^2$\!. In order to transition from~$a_t$ to
$a_{t+1}$, the algorithm requires $\bigO{n^2}$ time
(since a 2-SAT formula has $\bigO{n^2}$ distinct clauses), concluding the proof.
\end{proof}

\subsection{The Moran Process}\label{sec:moran}

The Moran process is a stochastic process introduced in biology to model the spread of genetic mutations in populations. In the Moran process, as introduced in~\cite{Mor1958:Moran}, we are given a population of $n$ individuals, each being one of two types: mutants and non-mutants. The process has a parameter $r \in \Re^+$, which is the fitness of mutants. All non-mutants have fitness~$1$. At each time step, one individual~$x$ is chosen for reproduction, with probability equal to its fitness over the total fitness of the population, i.e. the sum of the fitness of all individuals. The chosen individual then replaces another individual~$y$ of the population, chosen uniformly at random, with a new individual of the same type as~$x$. Thus, if~$x$ is a mutant, then it replaces~$y$ with a mutant, and if~$x$ is a non-mutant, it replaces~$y$ with a non-mutant. When running this process indefinitely with $n-1$ non-mutants and a single mutant as a starting state, it will either reach a state where the population consists exclusively of mutants, which we call \emph{fixation}, or reach a state where the population consists exclusively of non-mutants, which we call \emph{extinction}. One of the core purposes of the Moran process is to compute the fixation probability (or, equivalently, the extinction probability) of a population.

Lieberman, Hauert, and Nowak~\cite{LHN2005:EvoDyn} extended the original Moran process by introducing structured populations in the form of directed graphs. Earlier work on the \emph{time of absorption} (that is, either fixation or extinction) of the Moran process is about simple graphs, such as complete graphs, stars, and regular undirected graphs (see~\cite{BHRS10,TIN06}). A drift theorem was used by D{\'\i}az et al.~\cite{DGMRSS14} to show that, on any undirected graph, the expected time of absorption of the Moran process is $\bigO{n^6}$ when $r>1$, $\bigO{n^4}$ when $r=1$, and $\bigO{n^3}$ when $r<1$, and they also show concentration for these bounds. D{\'\i}az et al.~\cite{DGRS16} show that there are families of \emph{directed} graphs that have, in expectation, exponential time of absorption. Finally, Goldberg, Lapinskas and Richerby~\cite{DBLP:journals/corr/abs-1804-02293}, employed a more elaborate drift potential to show a tight upper bound in $\smallO{n^{3+\varepsilon}}$ of the absorption time for all positive $r\neq 1$.

We show how to use drift to obtain an upper bound on the time to absorption of the original Moran process with $n$ individuals. %In the proof, we will see how a clever potential can lead to very good bounds.
Let $Y_t$ be the number of non-mutants after $t \in \Na$ iterations of the process. Thus, $(Y_t)_{\timePoint \in \Na}$ is a random process on $[0, n] \cap \Na$ with starting value $n-1$ and absorbing states $0$ and $n$. For any $k \in [1, n - 1] \cap \Na$ and all $t \in \Na$, we have
\[
    \Pr{Y_t - Y_{t+1} = -1}[Y_t=k] = \frac{k}{k + r(n-k)} \cdot \frac{n-k}{n}\ .
\]
The first term of the product is because we have a chance of $k/\left(k+r(n-k)\right)$ for choosing a non-mutant for reproduction. The
second term is because there are $n-k$ mutants among $n$ individuals that can be replaced from a non-mutant. For all $k \in [1, n - 1] \cap \Na$, we abbreviate 
\[
p(k) = \frac{k}{k + r(n-k)} \cdot \frac{n-k}{n}
\]
as the probability that the number of non-mutants increases when we currently have $k$ non-mutants. Thus, we get
\begin{equation}\label{eq:moran-increase}
    \Pr{Y_t - Y_{t+1} = -1}[Y_t=k] = p(k).
\end{equation}
Using similar reasoning, we get
\begin{equation}\label{eq:moran-decrease}
    \Pr{Y_t - Y_{t+1} =   1}[Y_t=k] = \frac{r(n - k)}{k + r(n-k)} \cdot \frac{k}{n} = r \cdot p(k)\ .
\end{equation}

We now bound the expected absorption time of the process.

\begin{theorem}\label{thm:driftOnMoran}
     The expected time to absorption of the original Moran process on $n \in \Na+$ individuals is, \vspace{-.8cm}
    \[
     \hspace*{1.5 cm}
     \begin{tabular}{ll}
      \bigO{n^2} & \textrm{ when $r = 1$;}\\
      \bigO{\frac{r+1}{r-1}n\log n} & \textrm{ when $r > 1$; and}\\
      \bigO{\frac{1+r}{1 - r}n\log n} & \textrm{ when $r < 1$.}\end{tabular}
    \]
    \end{theorem}
\begin{proof}
    Recall that $Y_t$ is the number of non-mutants in the population after $t \in \Na$ iterations of the process. We first consider the \textbf{case $\boldsymbol{r=1}$}. From \Cref{eq:moran-increase,eq:moran-decrease}, for all $t \in \Na$ and all $k \in [1, n - 1] \cap \Na$, we have that
    \[
        \Pr{Y_{t+1} = Y_t + 1}[Y_t=k]  = \Pr{Y_{t+1} = Y_t - 1}[Y_t=k] = p(k)\ .
    \]
    We first give the following lower bound on the variance of the process while none of the absorbing states are found,
    \[
        \Var{Y_{t+1} - Y_t}[Y_t=k] = 2p(Y_t) \geq 2 \frac{n - 1}{n} \cdot \frac{1}{n}\ .
    \]
    Thus, using the above lower bound for the variance together with a version of \Cref{thm:variance-two-barrier-hitting-time} that works with such a bound (as mentioned after the theorem), we obtain an upper bound on expected first-hitting time of the absorbing states of $n^2/2$.

    We now consider the \textbf{case $\boldsymbol{r>1}$}.
    % From \Cref{eq:moran-increase,eq:moran-decrease} we have, for all $t \in \Na$ and all $k \in [1, n - 1] \cap \Na$, that $\E{Y_{t} - Y_{t+1}}[Y_t=k] = (r - 1) \cdot p(k)$.
    We aim to apply \Cref{thm:generalFiniteStateSpace}, however, we need to consider a process that is~$0$ once fixation or extinction occurred---as $Y_t$ is only~$0$ in the case of fixation. For this reason, we consider the process $(Y'_t)_{t \in \Na}$ which behaves exactly like~$Y_t$ but states $0$ and $n$ are both mapped to~$0$. That is, $(Y'_t)_{t \in \Na}$ operates on the states $[0, n - 1] \cap \Na$. Note that $T = \inf \{t \in \Na \mid Y'_t = 0\}$ denotes the absorption time of the Moran process.

    If $Y_t < n - 1$, then the expected change of~$Y'_t$ is the same as that of~$Y_t$. For $Y'_t = n - 1$, we get $\Pr{Y'_{t} - Y'_{t + 1} \geq 1}[Y'_t = n - 1] = r \cdot p(k) + p(k) \geq r \cdot p(k)$ as well as $\Pr{Y'_{t} - Y'_{t + 1} = -1}[Y'_t = n - 1] = 0 \leq p(k)$. We already argued that $\pplus(k) = r \cdot p(k)$ and $\pminus(k) = p(k)$ are choices that fulfill all requirements of the theorem. Thus, \Cref{thm:generalFiniteStateSpace} gives an upper bound on the time to absorption of
    \begin{align*}
    		\E{\firstHittingTime} & \leq \sum_{s = 1}^{n-1}\sum_{i = s}^{n-1}\frac{1}{r p(i)}\prod_{j = s}^{i - 1}\frac{p(j)}{r p(j)}
%     		 & = \sum_{s = 1}^{n-1}\sum_{i = s}^{n - 1}\frac{n}{r(n-i)}\cdot\frac{i+r(n-i)}{i} \left(\frac{1}{r}\right)^{i-s}\\
    		  = \sum_{i = 1}^{n-1}\sum_{s = 1}^{i}\frac{n}{r(n-i)}\cdot\frac{i+r(n-i)}{i} \left(\frac{1}{r}\right)^{i-s}\\
    		 & = \sum_{i = 1}^{n-1}\frac{n}{r(n-i)}\cdot\frac{i+r(n-i)}{i}\sum_{s = 1}^{i}\left(\frac{1}{r}\right)^{i-s}
    		  \leq \sum_{i = 1}^{n-1}\frac{n}{r(n-i)}\cdot\frac{i+r(n-i)}{i} \frac{r}{r-1}\\
    		 & =\frac{n}{r-1}\left(\sum_{i = 1}^{n-1}\frac{1}{n-i} + \sum_{i = 1}^{n-1}\frac{r}{i} \right) 
    		   \in \bigO{\frac{r+1}{r-1}n\log n} \ .
    \end{align*}

    Last, for the \textbf{case $\boldsymbol{r<1}$}, consider $Z_t = (n - Y_t) \cdot \boldsymbol{1}\{Y_t > 0\}$ and notice that for this process the states~$0$ and~$n$ from $(Y_t)_{t \in \Na}$ are both mapped to~$0$. From \Cref{eq:moran-increase,eq:moran-decrease}, for all $t \in \Na$ and $k \in [1, n - 1] \cap \Na$, we have
        $\Pr{Z_t - Z_{t+1} \geq 1}[Z_t=k] \geq p(k)$, as well as $\Pr{Z_t - Z_{t+1} = -1}[Z_t=k] \leq r \cdot p(k)\ .$
    Applying \Cref{thm:generalFiniteStateSpace} with $\pplus(k) = p(k)$ and $\pminus(k) = r \cdot p(k)$, after calculations similar to the $r>1$ case we have the following bound.
    \begin{align*}
     \E{\firstHittingTime} & \leq \sum_{s = 1}^{n-1}\sum_{i = s}^{n-1}\frac{1}{p(i)}\prod_{j = s}^{i - 1}\frac{r p(j)}{p(j)}
%      &= \sum_{s = 1}^{n-1}\sum_{i = s}^{n-1}\frac{n}{n-i}\cdot\frac{i+r(n-i)}{i} \cdot r^{i-s}\\
%      &= \sum_{i = 1}^{n-1}\sum_{s = 1}^{i}\frac{n}{n-i}\cdot\frac{i+r(n-i)}{i} \cdot r^{i-s}\\
%      &= \sum_{i = 1}^{n-1}\frac{n}{n-i}\cdot\frac{i+r(n-i)}{i}\sum_{s = 1}^{i}r^{i-s}\\
%      &\leq \frac{n}{1-r}\left( \sum_{i = 1}^{n-1}\frac{1}{n-i}+ \sum_{i = 1}^{n-1}\frac{r}{i} \right)\\
     \in \bigO{\frac{1+r}{1 - r}n\log n}\ .\qedhere
    \end{align*}
\end{proof}

%% file: conclusions.tex
%We presented drift theory: a collection of theorems that infer expected hitting times of randomized processes by only looking at local changes of the process. We stated the most common drift theorems and explained the intuition underlying drift theory, making the theorems easy to understand. We furthered this understanding by giving examples of how to apply them to classical randomized problems. In many of these cases, the application of a drift theorem was straightforward. However, we also considered more complicated examples, where a more elaborate potential function was necessary to get good bounds.

The beauty in working with drift theory is that it is very intuitive: Expected gains translate into expected times to gain a certain amount. Typically, the process at hand does not need to be fitted too much to suit the needs of a drift theorem, since different drift theorems are available for different needs. In the theory of randomized search heuristics, there are several instances where drift theory has been used to significantly simplify previous analyses.

We encourage our readers to apply drift theory themselves the next time they are facing a randomized process and the research question can be formulated as a hitting time. The biggest challenge then will most likely be to come up with potentials that are well-suited for the process of interest. In order to construct highly non-trivial potentials that end up yielding the desired results, a lot of trial and error is necessary~-- as always. We hope that we sparked the reader's interest in drift theory and wish them much success in applying it.

%% file: appendix.tex
We explain in more detail what our simplifications in the article mean from a more formal point of view.

\subparagraph{Conditioning on Events}
We condition on events where a random variable~$X$ takes on a universally quantified value~$s$.
In this quantification, we do not restrict the domain of~$s$ to only such values where the event $\{X = s\}$ has a positive probability, although this results in many undefined cases.
We use no restrictions on~$s$ in order to keep the notation shorter.
However, what we mean is to only consider such values of~$s$ where $\Pr{X = s} > 0$.

\subparagraph{Integrability}
A random variable~$X$ is \emph{integrable} if and only if $\E{|X|} < \infty$.
In general, a random process $(\randomProcess[\timePoint])_{\timePoint \in \Na}$ is \emph{integrable} if and only if, for all $\timePoint \in \Na$, it holds that \randomProcess[\timePoint] is integrable.

\subparagraph{Markov Chains}
A random process $(\randomProcess[\timePoint])_{\timePoint \in \Na}$ is a \emph{Markov chain} if and only if the outcome of each next step only depends on the current state and time point.
Formally, for all $\timePoint \in \Na$ as well as all $s \in \Re$ and all $\boldsymbol{v} \in \Re^\timePoint$, it holds that $\Pr{\randomProcess[\timePoint + 1] = s}[\randomProcess[\timePoint] = \boldsymbol{v}_\timePoint] = \Pr{\randomProcess[\timePoint + 1] = s}[\forall \timePoint' \in [0, \timePoint] \cap \Na\colon \randomProcess[\timePoint'] = \boldsymbol{v}_{\timePoint'}]$.

\subparagraph{Quantification over Random Variables}
We state inequalities that only need to hold for points in time when a random process did not reach its target yet.
Formally, let~\firstHittingTime be a random variable over $\Na \cup \{\infty\}$, let~$X$ and~$Y$ be random variables over~$\Re$, and for any event~$A$, let $\boldsymbol{1}\{A\}$ denote the indicator function for~$A$.
Note that~$\boldsymbol{1}\{A\}$ is a random variable.
Further, let~$\sim$ denote a relation symbol, such as $=$, $\leq$, or $\geq$.
We define the phrase ``for all $\timePoint < \firstHittingTime$, it holds that $X \sim Y$'' to be equivalent to ``for all $\timePoint \in \Na$, it holds that $X \cdot \boldsymbol{1}\{\timePoint < \firstHittingTime\} \sim Y \cdot \boldsymbol{1}\{\timePoint < \firstHittingTime\}$''.